\pdfoutput=1
\documentclass[a4paper]{article}
\usepackage[T2A]{fontenc}
\usepackage[cp1251]{inputenc}
\usepackage[english,russian]{babel}
\AtBeginDocument{%
  }
  \AtBeginDocument{%
  \renewcommand\refname{References}}
\usepackage[tbtags]{amsmath}
\usepackage{amsfonts,amssymb,mathrsfs,amscd}
 \usepackage{msb-a}
\JournalName{}
\numberwithin{equation}{section}
\theoremstyle{plain}
\newtheorem{conjecture}{Conjecture}
\newtheorem{theorem}{Theorem}
\theoremstyle{plain}
\newtheorem{theoremm}{Theorem}
\newtheorem{lemma}{Lemma}
\newtheorem{corollary}{Corollary}
\theoremstyle{definition}
\newtheorem{proof}{Proof}
\newtheorem{definition}{Definition}

\begin{document}

\title{On Nov{\'a}k numbers}
\author{Alexander~Kalmynin}
\address{National Research University Higher School of Economics, Math Department}
\email{alkalb1995cd@mail.ru}
\date{}
\udk{}
\maketitle
\begin{fulltext}
\begin{abstract}
In this work, we obtain some new lower bounds for the number $\mathcal N_B(x)$ of Nov{\'a}k numbers less than or equal to $x$.  We also prove, conditionally on Generalized Riemann Hypothesis, the upper estimates for the number of primes dividing at least one Nov{\'a}k number and give description for the prime factors of  Nov{\'a}k numbers $N$, such that $2N$ is a Nov{\'a}k-Carmichael number.
\end{abstract}
 
\section{Introduction}

Nov{\'a}k numbers are natural numbers $N$ such that $2^N+1$ is divisible by $N$. These numbers can be considered as some analogue of pseudoprimes to the base 2, that is natural numbers $N$ with the property that $N \mid 2^{N-1}-1$. In view of this analogy, one might ask why we should consider condition $N \mid 2^N+1$ and not $N \mid 2^N-1$. But it turns out that the latter case is trivial--- if $2^N-1$ is divisible by $N$, then $N$ necessarily equals $1$.

The set of Nov{\'a}k numbers possesses numerous additional structures. For example, unlike in the case of Rotkiewicz numbers (see \cite{Rot}), i.e. natural numbers $N$ with $2^{N-2}-1$ divisible by $N$, it is rather easy to see that there are infinitely many Nov{\'a}k numbers. Indeed, 3 is a Nov{\'a}k number and the set of Nov{\'a}k numbers is closed under gcd, lcm and multiplication (see section 3), so, for any $k \in \mathbb N$ we have $3^k \mid 2^{3^k}+1$. Also, if $N$ is a Nov{\'a}k number, then so is $M=2^N+1$. To prove this, note that $M=2^N+1=Nk$ for some positive integer $k$ and $k$ is odd since it is a factor of odd number $2^N+1$. So, the following equality holds:

$$2^M+1=2^{Nk}+1=(2^N+1)(2^{N(k-1)}-2^{N(k-2)}+\ldots+1)=ML,$$

\noindent where $L$ is an integer. We conclude that $M\mid ML=2^M+1$, as required.

In this paper, we study the distribution of Nov{\'a}k numbers, namely, we will obtain some new lower bounds for the number of Nov{\'a}k numbers less than a given number $x$. Here and further this quantity will be denoted $\mathcal N_B(x)$ in honor of Br{\v e}tislav Nov{\'a}k. Professor B.Nov{\'a}k was first to define the Nov{\'a}k numbers and to prove nontrivial bounds for $\mathcal N_B(x)$. He also computed the first million of Nov{\'a}k numbers and conjectured that $N_B(x)\ll x^\varepsilon$ for any $\varepsilon>0$ (for more details, see \cite{NovakOb1},\cite{NovakOb2}).

Paper \cite{Pomerance1} by Luca, Pomerance, Shparlinski and Alba Gonz{\'a}lez contains, among many other facts, the following theorem:

\begin{theoremm}
For some $c>0$ and all large enough $x$ we have

$$xe^{-(1+o(1))\sqrt{\ln x\ln\ln x}} \gg \mathcal N_B(x) \gg e^{c(\ln\ln x)^2}$$

\end{theoremm}

Main goal of our paper is to prove a much better lower bound.

\begin{theorem}

There exist a positive constants $c_2$ and $x_2$ such that for any $x>x_2$ the inequality

$$ \mathcal N_B(x) \gg e^{e^{c_2(\ln\ln\ln x)^2}}$$

is satisfied.

\end{theorem}

Next theorem shows that we can continue this process and make the tower of exponents arbitrarily large.

\begin{theorem}

For any positive integer $n$ there exist positive constants $c_n$ and $X_n$ such that for any $x>X_n$ the lower bound

$$ \mathcal N_B(x) \gg_n e_n(c_n(\ln_{n+1} x)^2)$$

holds. Here $e_0(x)=x=\ln_0(x)$ and $e_{i+1}(x)=e^{e_i(x)}$, $\ln_{i+1}(x)=\ln(\ln_{i}(x))$ for any $i \geq 0$.

\end{theorem}

Theorems 1 and 2 are also true for some rather general class of  sequences. For example, our considerations are still valid for such $N$ that $a^N-b^N$ is divisible by $N$, where $a$ and $b$ are fixed integers with $a-b \neq 0,\pm1$.

In the fourth section of this work we will discuss the distribution of Nov{\'a}k primes, that is primes $p$ such that there exists at least one Nov{\'a}k number $N$ with $p \mid N$. Derivation of upper bounds for number of Nov{\'a}k primes less than a given magnitude will be conditional on Generalized Riemann Hypothesis (see next section).

\section{Notation and lemmas}

In this section, we place some useful lemmas and notaton.

For rational $q$ and prime $p$, we will denote by $\nu_p(q)$ the $p$-adic valuation of $q$, that is integer $k$, such that $q=p^k\frac{a}{b}$, with $a$ and $b$ coprime integers not divisible by $p$. For natural $N$, $\tau(N)$ and $\omega(N)$ are number of divisors and number of different prime factors of $N$, respectively. The greatest common divisor and the least common multiple of two integers $M$ and $N$ will be denoted $(M,N)$ and $[M,N]$. If $p$ is a prime number and $a$ and $b$ are coprime integers not divisible by $p$, then $\ell_p(\frac{a}{b})$ is the smallest positive integer $k$ with

$$a^k \equiv b^k \pmod p.$$ 

The following two lemmas about divisibility of $a^n\pm b^n$ are very essential in our proof of Theorems 1 and 2:
\begin{lemma}[(Lifting The Exponent Lemma)]

Let $a$, $b$ be integers, $n$ be a positive integer and $p$ be a prime such that $p$ divides $a-b$, but $ab$ is not divisible by $p$. Then

$$\nu_p(a^{k}-b^{k})=\nu_p(a-b)+\nu_p(k)$$

\end{lemma}

\begin{lemma}[(Zsigmondy's Theorem for sums)]
 Let $a$, $b$ be different coprime natural numbers and $n$ be natural, greater than 1. Then there exists a prime divisor of $a^n+b^n$ that does not divide $a^k+b^k$ for all $k<n$, except for the case $(a, b, n) = (2, 1, 3)$.

\end{lemma}

For the proof of these lemmas, see \cite{Mich}.

We also need one fact about the distribution of $\ell_p(g)$:
\begin{lemma}

Let $g \neq 0,\pm 1$ be some fixed rational number. If Generalized Riemann Hypothesis is true, then for all $x$ and all $1\leq L \leq \frac{\ln x}{\ln\ln x}$ we have

$$\left|\left\{p \leq x: \ell_p(g) \leq \frac{p-1}{L}\right\}\right| \ll_g \frac{\pi(x)}{L}$$
\end{lemma}

Proof  of this proposition is given in \cite{Pomerance2}.

This lemma is the only statement in this paper, that relies on Generalized Riemann Hypothesis.

We give here also formulation of the Large Sieve inequality, which is convenient for further applications:

\begin{lemma}[(Large Sieve inequality)]

Let $N$ and $Q$ be natural numbers, $f(p) \in \mathbb N$ for all primes $p \leq Q$ and $0<f(p)<p$. For any such $p$, fix arbitrary $f(p)$ residue classes modulo $p$. Let $\mathcal A$ be the set of natural numbers not exceeding $N$ and not lying in any of fixed residue classes. Then the following estimate for the number of elements in $\mathcal A$ holds:

$$|\mathcal A| \ll \frac{N+Q^2}{S},$$

where $$S=\sum_{q\leq Q} \mu^2(q)\prod_{p \mid q} \frac{f(p)}{p-f(p)}$$.
\end{lemma}

\section{Proof of main results}

In this part of the work, we will prove Theorems 1 and 2.

It was already mentioned before that the set of Nov{\'a}k numbers carries a lot of structures.
For example, if $N$ and $M$ are Nov{\'a}k numbers, then so are $(N,M)$ and $[N,M]$. To prove this, note that if $N$ divides $2^N+1$ and $M$ divides $2^M+1$, then $N$ and $M$ are both odd and, consequently, $2^{(N,M)}+1=(2^N+1,2^M+1)$ and $2^{[N,M]}+1$ is divisible by $[2^N+1,2^M+1]$. But $2^N+1$ and $2^M+1$ are both divisible by $(N,M)$ and $[2^N+1,2^M+1]$ is divisible by $[N,M]$, and this completes the proof. It is a bit harder to prove that product of two Nov{\'a}k numbers is a Nov{\'a}k number. This fact can be deduced from the following more general statement:

\begin{lemma}

Let $N$ be a Nov{\'a}k number and $p_1,\ldots,p_k$ be some prime factors of $2^N+1$. Then for any nonnegative integers $\alpha_1,\ldots,\alpha_k$ the number $Np_1^{\alpha_1}\ldots p_k^{\alpha_k}$ is a Nov{\'a}k number.

\end{lemma}

\begin{proof}

First, we show that for any Nov{\'a}k number $N$ and any prime $p \mid 2^N+1$ the number $2^{Np}+1$ is divisible by $p^{\nu_p(N)+1}$. If $N$ is not divisible by $p$, then, as $p$ is odd and divides $2^N+1$, we have $2^N+1 \mid 2^{Np}+1$ and therefore $p$ divides $2^{Np}+1$. If $N$ is divisible by $p$, then $\nu_p(N)>0$ and, by Lifting the Exponent Lemma, $\nu_p(2^{Np}+1)=\nu_p(2^N+1)+1 \geq \nu_p(N)+1$.

Thus, $2^{Np}+1$ is divisible by $N$ and $p^{\nu_p(N)+1}$. Consequently, it is divisible by $[N,p^{\nu_p(N)+1}]=Np$. So, $Np$ is a Nov{\'a}k number. Repeatedly applying this proposition, we obtain the required result.
\end{proof}

Now, using Lemma 5, we will show that to prove good lower bounds for $\mathcal N_B(x)$ it is sufficient to construct Nov{\'a}k numbers $N$ with $2^N+1$ having many different prime factors.
\begin{lemma}

Let $x$ be a positive real number and $1<N \leq x$ be a Nov{\'a}k number with $\omega(2^N+1)=k$. Then the inequality
$$\mathcal N_B(x) \geq \left(\frac{\ln \frac{x}{N}}{N}\right)^k$$
holds.
\end{lemma}

\begin{proof}

If $p_1,\ldots,p_k$ are different prime factors of $2^N+1$, then, by the Lemma 5, any number of the form $Np_1^{\alpha_1}\ldots p_k^{\alpha_k}$ with $\alpha_i \geq 0$ is a Nov{\'a}k number. We will prove a lower bound for the number of $k$-tuples $(\alpha_1,\ldots,\alpha_k)$ such that the corresponding Nov{\'a}k number does not exceed $x$.
\\
Taking logarithms of both sides of inequality $Np_1^{\alpha_1}\ldots p_k^{\alpha_k} \leq x$, we get the condition

$$\alpha_1\ln p_1+\ldots+\alpha_k\ln p_k \leq \ln \frac{x}{N}.$$

Any $k$-tuple, satisfying the conditions $\alpha_i \ln p_i \leq \frac{\ln\frac{x}{N}}{k}$ for all $1 \leq i \leq k$ also obviously satisfies the prevous inequality. Therefore, there are at least $$\prod_{i=1}^k \left(\left[\frac{\ln{\frac{x}{N}}}{k\ln p_i}\right]+1\right) \geq \prod_{i=1}^k \frac{\ln{\frac{x}{N}}}{k \ln p_i}=\left(\frac{\ln{\frac{x}{N}}}{k}\right)^kS$$

suitable $k$-tuples, where $S=\prod_{i=1}^k \frac{1}{\ln p_i}$. It remains to estimate the quantity $S$. To do this, let's note that, as $2^N+1$ is divisible by $p_1\ldots p_k$ the inequality

$$\ln p_1+\ldots+\ln p_k \leq \ln(2^N+1) \leq N$$

holds. By AM-GM inequality, we conclude
$$\prod_{i=1}^k \ln p_i \leq \left(\frac{\ln p_1+\ldots+\ln p_k}{k}\right)^k \leq \frac{N^k}{k^k}$$

that is $$S \geq \frac{k^k}{N^k}$$.

Due to the previous considerations, we get

$$\mathcal N_B(x)  \geq \left(\frac{\ln \frac{x}{N}}{k}\right)^kS \geq \left(\frac{\ln \frac{x}{N}}{N}\right)^k,$$

as required.
\end{proof}

It seems reasonable to expect that for any $y$ there is a Nov{\'a}k number $N \leq y$ such that $2^N+1$ has \emph{normal} in order number of divisors, that is $\ln\ln{(2^y+1)} \sim \ln y$. This heuristic together with Lemma 6 and the choice $y=\sqrt{\ln x}$ gives us the lower bound

$$\mathcal N_B(x) \geq e^{c(\ln\ln x)^2},$$

which is as strong as Theorem 2 for the case $n=1$. So, we need to construct Nov{\'a}k numbers $2^N+1$ with \emph{abnormally large} number of prime factors.

The next simple consequence of Zsigmondy's theorem allows us to deduce some estimates for $\omega(2^N+1)$ in terms of arithmetical properties of $N$.

\begin{lemma}
For any odd positive integer $N$ the inequality
$$\omega(2^N+1) \geq \tau(N)-1$$
holds.
\end{lemma}

\begin{proof}

For any divisor $d \neq 3$ of the number $N$, we can choose, by Zsigmondy's theorem, some primitive prime factor $p_d$ of $2^d+1$. Note that, by primitivity, for $d' \neq d$ we always have $p_{d'} \neq p_d$. As $N$ is odd and for any $d$ we have $p_d \mid 2^d+1$, $2^N+1$ is divisible by $p_d$. Therefore, the number of different prime factors of $2^N+1$ is at least as large, as the number of divisors of $N$ that are not equal 3. So, this lemma is proved.
\end{proof}

Lemma 7 has a number of interesting consequences that are useful in the proof of our main theorem.

\begin{corollary}
For any $n \in \mathbb N$ we have
$$\omega(2^{3^n}+1) \geq n$$
\end{corollary}

\begin{corollary}

For any positive integers $n$ and $k$ the following lower bound holds:

$$\omega(2^{(2^{3^n}+1)^k}+1) \geq (k+1)^n-1 \geq k^n$$
\end{corollary}

\begin{proof}
First corollary is a straightforward consequence of Lemma 7 for the case $N=3^n$, as $\tau(3^n)=n+1$. The second one easily follows from the inequality $$\tau((2^{3^n}+1)^k) \geq (k+1)^{\omega(2^{3^n}+1)} \geq (k+1)^n$$

and the Lemma 7 with $N=(2^{3^n}+1)^k$.

\end{proof}

Now we are ready to prove Theorem 1.

\begin{proof}

Suppose that for some $n$ and $k$ we have $(2^{3^n}+1)^k=N \leq x$.

By the Lemma 6,
$$\mathcal N_B(x) \geq \left(\frac{\ln{\frac{x}{N}}}{N}\right)^{\omega(2^N+1)}$$

On the other hand, in virtue of Corollary 2, $\omega(2^N+1) \geq k^n$.

Therefore

$$\mathcal N_B(x) \geq \left(\frac{\ln{\frac{x}{N}}}{N}\right)^{k^n}.$$

Now, choosing $k=\left[\frac{\sqrt{\ln\ln x}}{2}\right]$ and $n=\left[\frac{\ln\ln\ln x}{2\ln 3}\right]$, we get

$3^n\leq \sqrt{\ln\ln x}$ and, consequently, $2^{3^n}+1 \leq e^{3^n} \leq e^{\sqrt{\ln\ln x}}$. Thus, $N \leq e^{k\sqrt{\ln\ln x}} \leq \sqrt{\ln x}$.

So, for large enough $x$ we have

$$\mathcal N_B(x) \geq \exp(k^n/3) \geq \exp(\exp((1/4\ln3+o(1))(\ln\ln\ln x)^2)),$$

which completes the proof of Theorem 1.

\end{proof}

Theorem 2 will be deduced from the Theorem 1.

For the convenience, we set

$$d(x):=\max\{\omega(2^N+1)| N\leq x, N \text{ is a Nov{\'a}k number}\}.$$

Lemma 6 implies:
\begin{corollary}
For any $x>15000$ the estimate

$$\mathcal N_B(x) \geq e^{d(\sqrt{\ln x})}$$

holds.
\end{corollary}

\begin{proof}

Suppose that $n\leq \sqrt{\ln x}$ is a Nov{\'a}k number such that $\omega(2^n+1)=d(\sqrt{\ln x})$. By the Lemma 6, we have

$$\mathcal N_B(X) \geq \exp(d(\sqrt{\ln x})\ln\left(\frac{\ln \frac{x}{n}}{n}\right)).$$

Noting that for $x \geq 15000$ the inequality

$$\ln\frac{\ln\frac{x}{n}}{n}\geq \ln(\sqrt{\ln x}\left(1-\frac{\ln\ln x}{2\ln x}\right))=\frac{1}{2}\ln\ln x+\ln\left(1-\frac{\ln\ln x}{2\ln x}\right)>1$$

is satisfied, we arrive at required proposition.
\end{proof}

This proves that Theorem 2 is implied by the analogous fact about $d(x)$:
\begin{theorem}

For any positive integer $n$ there exist positive constants $c_n$ and $X_n$ such that for any $x>X_n$ the lower bound

$$d(x) \gg_n e_n(C_n(\ln{_{n+1} x})^2)$$

holds.
\end{theorem}

\begin{proof}

We remark that from the proof of Theorem 1 follows the truth of Theorem 3 in the case $n=1$ (for any $C_1<\frac{1}{4\ln 3}$).
Let's prove Theorem 3 by induction on $n$.

Suppose that

$$d(x) \gg e_k(C_n(\ln_{k+1} x)^2)$$

Using Lemma 7 and the fact that $2^N+1$ is Nov{\'a}k number if  $N$ is, we get

$$d(2^x+1) \geq 2^{d(x)}-1$$

Indeed, if $n\leq x$ is a Nov{\'a}k number with $\omega(n)=d(x)$, then $2^n+1\leq 2^x+1$  is a Nov{\'a}k number, too, and by the Lemma 7,

$$\omega(2^n+1) \geq \tau(n)-1 \geq 2^{\omega(n)}-1=2^{d(x)}-1.$$
So, the inequality

$$d(x) \gg 2^{d(\log_2(x-1))}-1 \gg_k 2^{e_k(C_k(\ln_{k+2}(x)))}-1$$

holds. Therefore, it is sufficient to take any $C_{k+1}<C_k$.
\end{proof}

Theorem 3 is proved and, thereby, Theorem 2 too.

Our theorems can be generalized to the special class of the so-called divisibility sequences.
\begin{definition}

Let $\mathcal U=\{u_n\}$  be a sequence of integers. $\mathcal U$ is admissible divisibility sequence of moderate growth if the following 5 properties are satisfied:

\begin{enumerate}
    \item{\bf{Divisibility:}} If $n \mid m$, then $u_n \mid u_m$
    \item{\bf{Lifting the Exponent Lemma:}} If $p \mid u_n$, then $pu_n \mid u_{pn}$.
    \item{\bf{Zsygmondy-Banks property:}} For all positive integers $n$ with finite number of exceptions there exists a prime $p$ such that  $p \mid u_n$, but $p \nmid u_k$ for all $k<n$.
    \item{\bf{Moderate growth:}} There exists $a>0$ such that $u_n \ll a^n$.
    \item{\bf{Nondegeneracy:}} $u_1 \neq \pm 1$.
\end{enumerate}
\end{definition}

Using the same considerations as in the proofs of theorems 1 and 2, the following fact is easily deduced:

\begin{theorem}
For $\mathcal U$ is an admissible divisibility sequence of moderate growth define $\mathfrak U:=\{n:n\mid u_n\}$. Let $U(x)=|\mathfrak U\cap [1,x]|$ be the corresponding counting function. Then for any positive integer $n$ there exist two positive constants $x(n,\mathcal U)$ and $c(n,\mathcal U)$, such that for any $x>x(n,\mathcal U)$ we have

$$U(x) \gg_n e_n(c(n,\mathcal U)(\ln{_{n+1} x})^2).$$
\end{theorem}

Moderate growth condition can be replaced by weaker estimate (for example, by a bound of the form $u_n \ll a^{n^2}$), but this replacement will affect the right hand side of the inequality. Nondegeneracy condition cannot be weakened, as the sequence $\{2^n-1\}$ satisfies all the conditions, except nondegeneracy, and has $U(x) \equiv 1$.

\section{Nov{\'a}k primes}

Prime number $p$ is called a Nov{\'a}k prime if there exists some Nov{\'a}k number $N$ that is divisible by $p$. Let us denote the set of all Nov{\'a}k primes by $\mathcal P_{\mathcal N}$. We will also need to consider the set of all natural numbers, having no prime factor inside $\mathcal P_{\mathcal N}$. This set will be denoted $\overline{P_{\mathcal N}}$. In this section, we are interested in the distribution of Nov{\'a}k primes, that is in growth rate of the function

$$\pi_{\mathcal N}(x)=|\mathcal P_{\mathcal N} \cap [1,x]|.$$

We expect that Nov{\'a}k primes are rare. For example, the number 9137 is only the seventh element of $\mathcal P_{\mathcal N}$. Main goal of this part of the work is the proof of the following statement:
\begin{theorem}
If Generalized Riemann Hypothesis is true, then the following inequality is satisfied
$$\pi_{\mathcal N}(x) \ll \frac{x{\ln{\ln x}}}{(\ln x)^2}.$$
\end{theorem}

To prove this sort of an estimate, we need to find some restrictions for the set of Nov{\'a}k primes. To guess, what we can do, let's look at the prime factors of first 24 numbers of the form $p-1$ with $p \in \mathcal P_{\mathcal N}$:
\vspace{0.2cm}
\begin{center}
\begin{table}[h]
\begin{tabular}{|c|c||c|c||c|c|}
\hline
$p$ & $p-1$ & $p$ & $p-1$ & $p$ & $p-1$\\
\hline
3 & $\boldsymbol{2}$ & 41113 &$\boldsymbol{2}^3\cdot\boldsymbol{3}^2\cdot\boldsymbol{571}$ & 174763 & $\boldsymbol{2}\cdot \boldsymbol{3}^2 \cdot 7 \cdot \boldsymbol{19} \cdot 73$ \\
\hline
19 & $\boldsymbol{2}\cdot\boldsymbol{3}^2$ & 52489 & $\boldsymbol{2}^3\cdot\boldsymbol{3}^8$ & 196579 & $\boldsymbol{2} \cdot \boldsymbol{3}^2 \cdot 67 \cdot \boldsymbol{163}$ \\
\hline
163 & $\boldsymbol{2}\cdot\boldsymbol{3}^4$ & 78787 & $\boldsymbol{2}\cdot\boldsymbol{3}^3\cdot\boldsymbol{1459}$ & 274081 & $\boldsymbol{2}^5 \cdot \boldsymbol{3} \cdot 5 \cdot \boldsymbol{571}$\\
\hline
571 & $\boldsymbol{2}\cdot\boldsymbol{3}\cdot5\cdot\boldsymbol{19}$ & 87211 & $\boldsymbol{2}\cdot\boldsymbol{3}^3\cdot5\cdot17\cdot\boldsymbol{19}$ & 370009 & $\boldsymbol{2}^3 \cdot \boldsymbol{3}^4 \cdot \boldsymbol{571}$\\
\hline
1459 & $\boldsymbol{2}\cdot \boldsymbol{3}^7$ & 135433 & $\boldsymbol{2}^3\cdot\boldsymbol{3}^4\cdot11\cdot\boldsymbol{19}$ & 370387 & $\boldsymbol{2}\cdot \boldsymbol{3}^3 \cdot \boldsymbol{19}^3$\\
\hline
8803 & $\boldsymbol{2}\cdot\boldsymbol{3}^3\cdot\boldsymbol{163}$ & 139483 & $\boldsymbol{2}\cdot\boldsymbol{3}^5\cdot7\cdot41$ & 478243 & $\boldsymbol{2} \cdot \boldsymbol{3}^2 \cdot \boldsymbol{163}^2$\\
\hline
9137 & $\boldsymbol{2}^4\cdot\boldsymbol{571}$ & 144667 & $\boldsymbol{2}\cdot\boldsymbol{3}^4\cdot\boldsymbol{19}\cdot47$ & 760267 & $\boldsymbol{2} \cdot \boldsymbol{3}^4 \cdot 13 \cdot \boldsymbol{19}^2$\\
\hline
17497 & $\boldsymbol{2}^3\cdot\boldsymbol{3}^7$ & 164617 & $\boldsymbol{2}^3\cdot\boldsymbol{3}\cdot\boldsymbol{19}^3$ & 941489 & $\boldsymbol{2}^4\cdot \boldsymbol{19}^2\cdot \boldsymbol{163}$\\
\hline
\end{tabular}
\caption{Factorizations of $p-1$'s. Elements of $\{2\}\cup \mathcal P_{\mathcal N}$ are made bold.}
\end{table}
\end{center}

Data of Table 1 allow us to observe that, heuristically, for most  $p \in \mathcal P_{\mathcal N}$  most of the prime factors of $p-1$ are in $\{2\}\cup\mathcal P_{\mathcal N}$. 

Lemma 3 gives some explanation of this phenomenon. Indeed, if $Np$ is a Nov{\'a}k number, then $p$ divides $2^{Np}+1$, consequently, $2^{2Np}-1$ is divisible by $p$ and $\ell_p(2) \mid 2Np$. On the other hand, if Generalized Riemann Hypothesis is true, then, for most primes $p$, the number $\ell_p(2)$ is a large prime factor of $p-1$. But all the prime factors of $2Np$ are by definition in $\{2\}\cup \mathcal P_{\mathcal N}$, therefore the same fact is true for most prime factors of most numbers of the form $p-1$, where $p \in \mathcal P_{\mathcal N}$.

To deduce the statement of Theorem 5 from the Lemma 3 rigorously, we need to use one lemma about free multiplicative subsemigroups in $\mathbb N$.

\begin{lemma}
Let $\mathcal Q$ be the set of prime numbers, such that for some positive $a$ and $b$ we have

$$\pi_{\mathcal Q}(x):=|\mathcal Q\cap [1,x]|=\frac{ax}{\ln x}+O(\frac{x}{\ln^{1+b} x})$$

If $\mathfrak Q$ is free subsemigroup of $\mathbb N$, generated by $\mathcal Q$, that is the set of all positive integers, having prime factors only in $\mathcal Q$, then we have the following asymptotic formula:

$$Q(x):=|\mathfrak Q\cap [1,x]|=C(\mathcal Q)x\ln^{a-1} x\left(1+O\left(\frac{1}{\ln\ln^{\min(1,b)} x}\right)\right)$$
\end{lemma}

\begin{proof} It is an easy consequence of Bredikhin theorem (see \cite{Post}, p.135).
\end{proof}

\begin{corollary}
Under the assumptions of the Lemma 8, we have

$$S(x,\mathfrak Q):=\sum_{\substack{n \in \mathfrak Q\\ n \leq x}} \frac{1}{n} \gg \ln^a x$$
\end{corollary}

Using these facts together, we can now prove Theorem 5:
\begin{proof}

The set $\mathcal P_{\mathcal N}$ can be represented as a disjoint union of two sets $\mathcal R$ and $\mathcal Q$. The set $\mathcal R$ consists of primes $p$ in $\mathcal P_{\mathcal N}$ with $\ell_p(2) \leq \frac{p\ln\ln p}{\ln p}$, and $\mathcal Q$ is the set of all the other primes in $\mathcal P_{\mathcal N}$. By the Lemma 3, if Generalized Riemann Hypothesis is true, then the following estimate

$$|\mathcal R\cap[1,x]|\ll \frac{x\ln\ln x}{\ln^2 x}$$

holds. For all the other primes in $\mathcal P_{\mathcal N}$ we have $\ell_p(2) | p-1$, all the prime factors of $\ell_p(2)$ are either 2 or Nov{\'a}k primes. Therefore, any $p \in \mathcal Q$ with $\sqrt x<p \leq x$ has the following two properties: first, for any prime $q \leq \sqrt x$, $p$ is not divisible by $q$ and secondly, if an odd prime $q \not \in \mathcal P_{\mathcal N}$, but $p-1$ is divisible by $q$, then necessarily $q \leq \frac{\ln x}{\ln\ln x}$. Indeed, since all the odd prime factors of $\ell_p(2)$ are Nov{\'a}k primes, $q$ cannot divide $\ell_2(p)$. Thus, $p-1$ is divisible by $[q,\ell_2(p)]=q\ell_2(p)$. Consequently, $q \leq \frac{p-1}{\ell_2(p)}\leq \frac{\ln p}{\ln\ln p} \leq \frac{\ln x}{\ln\ln x}$.

So, $|\mathcal Q\cap(\sqrt x,x]| \leq |\mathcal A|$, where $\mathcal A$ is set of positive integers, not exceeding $x$, and not lying in any of $f(p)$ residue classes modulo $p$, where $p$ runs through all primes not exceeding $\sqrt x$. Here $f(p)$ is defined as follows: $f(p)=1$ if $p \leq \frac{\ln x}{\ln\ln x}$ or $p \in \mathcal P_{\mathcal N}$, and $f(p)=2$ otherwise.

Using Large Sieve inequality, we obtain the inequality

$$|\mathcal Q\cap[1,x]|\ll \sqrt x+\frac{x}{S},$$

with $$S=\sum_{n \leq \sqrt x} \mu^2(n)\prod_{p|n} \frac{f(p)}{p-f(p)}.$$

Let $Q(x)$ be subsemigroup, spanned by primes that are not in $\mathcal P_{\mathcal N}$ or not exceed $\frac{\ln x}{\ln\ln x}$ and $t(n)$ be the number of divisors of $n$, lying in $Q(x)$. It is easy to see that

$$S\geq \sum_{n \leq \sqrt x} \frac{\mu^2(n)t(n)}{n} \gg \sum_{n \leq \sqrt x} \frac{t(n)}{n}.$$

On the other hand,

$$\frac{t(n)}{n}=\sum_{\substack{mk=n\\ k\in Q(x)}} \frac{1}{km}$$

Therefore, the following inequality holds:

$$S \gg \sum_{m \leq x^{1/4}} \frac{1}{m} \sum_{\substack{k \leq x^{1/4} \\ k \in Q(x)}}\frac{1}{k} \gg \ln x \sum_{\substack{k \leq x^{1/4} \\ k \in Q(x)}}\frac{1}{k}$$

Furthermore,

$$\sum_{\substack{k \leq x^{1/4} \\ k \in Q(x)}}\frac{1}{k} \gg \sum_{\substack{k \leq x^{1/4} \\ k \in \overline{P_{\mathcal N}}}} \frac{1}{k} (L(x))^{-1},$$

where $L(x)$ is the sum of reciprocals of all positive integers with no prime factors greater than $\frac{\ln x}{\ln\ln x}$. Obviously, we have:

$$L(x)=\prod_{p \leq \frac{\ln x}{\ln\ln x}} \frac{p}{p-1} \asymp \ln\ln x$$

Thus, we finally get

$$S \gg \frac{\ln x}{\ln\ln x}\sum_{\substack{k \leq x^{1/4} \\ k \in \overline{P_{\mathcal N}}}}\frac{1}{k}$$

And deduce the following estimate

$$\pi_{\mathcal N}(x) \ll \frac{x\ln\ln x}{\ln^2 x}+\frac{x\ln\ln x}{\ln x}\left(\sum_{\substack{k \leq x^{1/4} \\ k \in \overline{P_{\mathcal N}}}}\frac{1}{k}\right)^{-1}$$

On the other hand, if $p \in \mathcal P_{\mathcal N}$ then there exists an odd number $N$ such that $p$ divides $2^N+1$, so $-2$ is a quadratic residue modulo $p$. Using this and Lemma 8, we prove that the sum at the right hand side of the latter inequality is at least $\gg\sqrt{\ln x}$, thus getting

$$\pi_{\mathcal N}(x) \ll \frac{x\ln\ln x}{\ln^{3/2} x}$$

Now, applying Corollary 4 again together with this estimate, we prove that the sum is $\gg \ln x$, so
$$\pi_{\mathcal N}(x) \ll \frac{x\ln\ln x}{\ln^2 x}$$

which proves Theorem 5.
\end{proof}

\section{Nov{\'a}k-Carmichael numbers and some conjectures}

Let's recall that natural number $N$ is called a \emph{Carmichael number} if $N$ divides $a^{N-1}-1$ for any integer $a$ coprime to $N$. It is a well-known fact that the following criterion holds:

\begin{theoremm}[(Korselt, 1899)]
A positive integer $n$ is a Carmichael number if and only if $n$ is squarefree and $p-1$ divides $n-1$ for any prime $p$ dividing $n$.
\end{theoremm}

In the frame of our work it is natural to consider an analogue of Carmichael numbers, i.e. positive integers $N$ with $a^N-1$ divisible by $N$ for any $a$ coprime to $N$. We will call these numbers a \emph{Nov{\'a}k-Carmichael numbers}. It's rather easy to see that an analogue of Korselt's criterion also holds (note that in this case $n$ is not necessary squarefree):

\begin{theorem}

A positive integer $n$ is a Nov{\'a}k-Carmichael number if and only if for all prime factors $p$ of $n$, it is true that $(p-1)\mid n$.
\end{theorem}

For example, 220 is a Nov{\'a}k-Carmichael, because it is divisible by $2-1$, $5-1$ and $11-1$.
It is clear that all the Nov{\'a}k-Carmichael numbers that are different from 1, are even. In view of this fact, it is interesting to ask, for which Nov{\'a}k numbers $N$ are $2N$ Nov{\'a}k-Carmichael? It turns out that prime factors of these numbers satisfy some very strong restrictions:

\begin{theorem}
Denote by $P_0$ the set of all prime numbers, that are congruent to 3 modulo 8.
For a positive integer $n$ we recursively define the set $P_n$ as the set of all primes $p \in P_{n-1}$ such that all the prime factors of $\frac{p-1}{2}$ are also in $P_{n-1}$. Denote by $P_\infty$ the intersection of all $P_n$'s for $n \geq 0$. Then a prime number $p$ is a prime factor of some Nov{\'a}k number $N$ such that $2N$ is a Nov{\'a}k-Carmichael number if and only if $p$ is an element of $P_{\infty}$.
\end{theorem}

\begin{proof}
We will prove our theorem by induction. First of all, note that $p$ is in $P_0$: indeed, $\left(\frac{-2}{p}\right)=1$, so $p$ is congruent to either 1 or 3 modulo 8. But it cannot be congruent to 1, as in this case, by our criterion, we have $8 \mid (p-1) \mid 2N$, which is a contradiction as $N$ is odd.

Assume that we have proved that every such prime $p$ is in $P_n$. By the Theorem 6, $p-1$ divides $2N$ and, consequently, $\frac{p-1}{2} \mid N$. So, by our assumption, every prime factor of $p-1$ is in $P_n$. This proves the first part of proposition.

Now, if $p \in P_\infty$, let's define the sequence $a_n$ by the formulas $a_0=p$, $a_{n+1}=[a_n,p_{1n}-1,p_{2n}-1,\ldots]$, where $p_{1n},p_{2n}\ldots$ are the prime factors of $a_n$. It is clear that $a_n$ stabilizes. Indeed, any $a_n$ is a divisor of $p!$ and $\{a_n\}$ is nondecreasing. Let $A=\lim_{n \to \infty} a_n$. Then, by Theorem 6, $A$ is a Nov{\'a}k-Carmichael number. On the other hand, every prime factor of $A$ is either $2$ or an element of $P_0$, because $a_{n+1}$ is always equals a least common multiple of $a_n$ and some number, having all prime factors in $P_{\infty}\cup \{2\}$. From this consideration it is also easily seen that $\nu_2(A)=1$. So, $A$ is of the form $2N$ for some $N$ with all prime factors in $P_0$. Thus, $-2$ is a quadratic residue modulo $N$. Since $N$ is odd, there exists some odd $m$ with $m^2 \equiv -2 \pmod N$.  Therefore, we have $(m,2N)=(m,A)=1$. Since $2N$ is a Nov{\'a}k-Carmichael number, we deduce that $m^A \equiv 1 \pmod N$ and, consequently, $-2^N=(-2)^N\equiv m^{2N}=m^A \equiv 1 \pmod N$. Thus, $N$ is a Nov{\'a}k number, which was to be proved.
\end{proof}

Based on the data of Table 1, we see that first few elements of $P_{\infty}$ are

$$3,19,163,1459,8803,78787,370387,478243\ldots$$

So, we expect that $P_{\infty}$ is a very thin, but yet infinite set of prime numbers.
\begin{conjecture}
The set $P_{\infty}$ is infinite.
\end{conjecture}

It might be interesting to formulate some heuristic arguments predicting the growth rate of the counting function of the set $P_{\infty}$. The problem of proving that at least some of the sets $P_n$ with $n>1$ are infinite is also seems to be a worthwhile question.
\end{fulltext}


\begin{thebibliography}{99}
\bibitem[1]{Rot} A. Rotkiewicz, On the congruence $2^{n-2} \equiv 1 \pmod n$, Math. Comput., 43
\bibitem[2]{NovakOb1}A.A. Karatsuba, {\v S}. Porubsk{\'y}, M. Rokyta, Z.Vl{\'a}{\v s}ek, Nedo{\v z}it{\'e} sedmdes{\'a}tiny Prof. RNDr. B{\v r}etislava Nov{\'a}ka, DrSc. (1938-2003), Pokroky matematiky, fyziky a astronomie, ro{\v c}nik 53 (2008), {\v c} 1, 53-58
\bibitem[3]{NovakOb2} A.A. Karatsuba, {\v S}. Porubsk{\'y}, M. Rokyta, Z.Vl{\'a}{\v s}ek, Prof. RNDr. B{\v r}etislav Nov{'a}k, DrSc. (1938-2003) would be seventy, Mathematica Bohemica, 133 (2008), 2, 209-218
\bibitem[4]{Pomerance1} J. J. Alba Gonz{\' a}lez, F. Luca, C. Pomerance, and I. E. Shparlinski, On numbers n dividing the n-th term of a linear recurrence,  Proc. Edinburgh Math. Soc., 55 (2012), 271-289.
\bibitem[5]{Mich} B. Michels, Zsigmondy's Theorem, 
users.ugent.be/\textasciitilde bmichels/files/ zsigmondy\_en.pdf
\bibitem[6]{Pomerance2} P. Kurlberg and C. Pomerance, On a problem of Arnold: the average multiplicative order of a given integer, Algebra and Number Theory, 7 (2013), 981-999.
\bibitem[7]{Post} A.G.Postnikov, Introduction to analytic number theory (in russian), Publ. <<Nauka>>, Moscow, 1971.
(1984)  MR 85e:11005.
\end{thebibliography}
\end{document}